\documentclass[11pt]{amsart}
\usepackage[english]{babel}
\usepackage{graphicx} 
\usepackage{amsmath,amsfonts,amssymb,amsthm,amscd} 
\usepackage{mathtools} 
\usepackage{tikz-cd} 
\usepackage{aliascnt} 
\usepackage{soul,color} 
\usepackage{mathrsfs}
\usepackage{eucal}
\usepackage{hyperref}

\usepackage[a4paper,margin=3cm]{geometry}

\usepackage{titlesec}

\titleformat{\section}
  {\normalfont\large\bfseries}
  {\thesection}{1em}{}

\titleformat{\subsection}
  {\normalfont\normalsize\bfseries}
  {\thesubsection}{1em}{}

\newtheorem{theorem}{Theorem}[section]
\newtheorem{proposition}[theorem]{Proposition}
\newtheorem{lemma}[theorem]{Lemma}

\theoremstyle{definition}
\newtheorem{definition}[theorem]{Definition}
\theoremstyle{remark}
\newtheorem{remark}[theorem]{Remark}

\def\Pset{\mathbb{P}}

\def\Qset{\mathbb{Q}}

\DeclareMathOperator{\Pic}{Pic}
\DeclareMathOperator{\Div}{Div}
\DeclareMathOperator{\supp}{supp}

\DeclareMathOperator{\Sing}{Sing}

\title{Generalizations of the Scorza correspondence}
\author{Maria Teresa Ruggiero}
\thanks{The author is a member of GNSAGA (INdAM)}
\date{}

\begin{document}

\begin{abstract}
We study generalizations of the Scorza correspondence associated with a general $(C,\eta)$ in the moduli space of even spin curves $S_g^+$. We first give a new proof of the smoothness of the classical Scorza curve, originally established by Farkas--Verra.

We then introduce higher order analogues of the Scorza correspondence and investigate their geometry. In particular, we compute their classes in the Néron--Severi group and study their singularities and genera.

Finally, we consider a higher-dimensional version of the construction and focus on the threefold case, where we describe its class and prove a smoothness result outside a naturally defined degeneracy locus.
\end{abstract}


\maketitle 



\section*{Introduction}

The geometry of correspondences associated with theta characteristics has long played a central role in the study of algebraic curves. An important object in this context is the \emph{Scorza
correspondence}: given a smooth complex projective curve $C$ of genus $g \geq 2$
and an ineffective even theta characteristic $\eta$, i.e. a line bundle $\eta\in\Pic^{g-1}(C)$ with $\eta^{\otimes 2} \cong \omega_C$ such that  $h^0(C,\eta)=0$, one defines
\[
    \Gamma_\eta := \{(p,q) \in C \times C \mid h^0(C,\, \eta \otimes \mathcal{O}_C(p-q)) > 0\}
    \subset C \times C.
\]

This curve, introduced by Scorza in \cite{Scorza1899,Scorza1900}, encodes subtle properties of
the spin structure of $C$ and has been the subject of extensive investigation \cite{DolgachevKanev1993,FarkasVerra2014,FarkasIzadi2024,GrushevskySalvatiManni2010}.The numerical class of $\Gamma_\eta$ in the N\'eron--Severi group of $C \times C$ is given by
$$
    \Gamma_\eta \equiv (g-1)C_1 + (g-1)C_2 + \Delta,
$$
where $C_i \in H^2(C\times C,\Qset )$ denotes the fiber class on the two projections on each factor of $C\times C$ and $\Delta$ denotes the diagonal.  Moreover, its arithmetic genus is $3g(g-1)+1$ (see \cite{DolgachevKanev1993}).

The smoothness of $\Gamma_\eta$ for a
general even spin curve was established by Farkas and Verra \cite{FarkasVerra2014} via a
degeneration argument to the boundary of the moduli space.

Motivated by the rich geometry of the classical Scorza correspondence, the aim of
this paper is to investigate how the spin structure of a curve gives rise to a
broader family of geometric objects, ranging from higher-order correspondences to
higher-dimensional varieties naturally associated with the theta characteristic.
Our guiding principle is that many of the phenomena governing the classical
Scorza curve admit natural
higher-dimensional analogues, whose geometry can still be studied through the
interaction between linear systems and spin structures.

To this aim, the present paper pursues three interrelated goals.

\medskip
\noindent\textbf{Scorza Correspondence $\Gamma_\eta$.}
The first part of the paper is devoted to the classical Scorza correspondence. After reviewing the necessary background and some known results from the literature, we present our new contributions concerning the smoothness of $\Gamma_\eta$ and the geometry of a particular linear system, strictly related to the construction of this correspondence.
\begin{itemize}
\item \emph{New proof of smoothness of $\Gamma_{\eta}$:} Our first result is a new proof of the Farkas--Verra smoothness theorem
(Theorem~\ref{liscezza}), which entirely avoids degeneration. Instead, we
combine the classical criterion of Griffiths--Harris \cite{GriffithsHarris1978} for coincident points of a correspondence with a vanishing result for
sections of the canonical bundle due to Fassina--Pirola \cite{FassinaPirola2026}. In our opinion, this
approach is both conceptually transparent and, crucially, flexible enough to handle
the higher-order constructions introduced below.
\item \emph{Self composition of $\Gamma_{\eta}$:} We study the geometry of the locus of pairs $(p,q)$ for which the linear system
$|\eta\otimes\mathcal O_C(p+q)|$ admits a base point. We show that this locus defines a curve
naturally described in terms of the self-composition of the Scorza correspondence, in the sense of Definition~\ref{def:composition}, namely
$$
B=\Gamma_\eta\circ\Gamma_\eta-g\Delta
$$
(Proposition~\ref{baselocus}), and compute its numerical class.
\end{itemize}
\medskip
\noindent\textbf{Higher order Scorza correspondences.}
The second part of the paper introduces the \emph{$m$-Scorza correspondence}
(Definition~\ref{nScorza}): for $m \geq 1$, we set
\[
    \Gamma_\eta^{m} := \{(p,q) \in C \times C \mid h^0(\eta \otimes \mathcal{O}_C(mp - mq)) > 0\}.
\]
This defines a curve which is a natural generalization of the classical Scorza curve, which corresponds to the case
$m=1$. Our main results are:
\begin{itemize}
    \item \emph{Numerical class:} We prove that $\Gamma_\eta^m \equiv m^2[(g-1)C_1 + (g-1)C_2 + \Delta]$,
 via the Pl\"ucker formula for ramification of the linear system
    $|\eta \otimes \mathcal{O}_C(mp)|$.
    \item \emph{Arithmetic genus:} $p_a(\Gamma_\eta^m) = (2m^2 + m^4)g(g-1) + 1$,
    obtained from the adjunction formula.
    \item \emph{Smoothness:} For a general $(C,\eta) \in S_g^+$ with $g \geq 2$,
    the curve $\Gamma_\eta^m$ ($m\geq 2$) is smooth away from the intersection locus
    $\Gamma_\eta^m \cap \Gamma_\eta^{m-1}$ (Theorem~\ref{liscezza2}).
    Singularities may arise from special points of the linear system
    $|\eta \otimes \mathcal{O}_C(mp)|$, which force higher ramification of the associated
    morphism $C \to \mathbb{P}^{m-1}$.
\end{itemize}

\medskip
\noindent\textbf{The Scorza threefold.}
In the third part we study a natural higher-dimensional analogue. Given a
general even ineffective spin curve $(C,\eta)$, for any integer $n \geq 1$, we consider the locus
$$
\Gamma_n = \Big\{ (p_1,\dots,p_n,q_1,\dots,q_n) \in C^{2n} \;\Big|\; 
h^0\big(C,\eta \otimes \mathcal{O}_C(\textstyle{\sum_{i=1}^n} p_i - \sum_{i=1}^n q_i)\big) > 0 \Big\}
$$
which defines a natural hypersurface in $C^{2n}$.

The higher-dimensional extension of the classical Scorza
correspondence has recently been considered by Agostini, Kummer, and
Park~\cite{AKP26}. Their construction is formulated on symmetric products
of the curve: starting from a non-effective theta characteristic, they
introduce higher Szeg\H{o} kernels whose zero divisors define higher Scorza
correspondences. These kernels are then used to construct symmetric admissible
determinantal representations and symmetric Ulrich sheaves on
 higher secant varieties. 

We focus, in particular, on the first nontrivial case beyond the classical one, namely $n=2$, which leads to a threefold $\Gamma:=\Gamma_2$ of $C \times C \times C \times C$.
Our main results are:
\begin{itemize}
    \item \emph{Class in the N\'eron--Severi group:}
    We prove that the class of $\Gamma$ has the following form
    \[
        [\Gamma] = (g-1)\sum_{i=1}^4 F_i - (\Delta_{12}+\Delta_{34})
        + (\Delta_{13}+\Delta_{14}+\Delta_{23}+\Delta_{24}),
    \]
    where $F_i = \pi_i^*[\mathrm{pt}]$ and $\Delta_{ij} = \pi_{ij}^{-1}(\Delta)$. The computation exploits the symmetry group
    $D_4$ acting on $C^4$ and reduces to three intersection
    numbers computed via ramification of suitable linear systems.
    \item \emph{Smoothness:} For a general $(C,\eta) \in S_g^+$ with $g \geq 2$,
    the threefold $\Gamma$ is smooth outside the locus
    $Z' = \{h^0(\eta\otimes\mathcal{O}_C(p+q-r-s)) = 2\}$ (Theorem~\ref{liscezza3}).
    The proof is based on a simultaneous ramification argument extending Griffiths--Harris' classical analysis, together with the vanishing criterion of \cite{FassinaPirola2026}.
\end{itemize}
More broadly, the results of this paper indicate that the geometry of spin curves
continues to produce a rich hierarchy of naturally associated varieties, whose
structure reflects in a subtle way the interplay between theta characteristics,
ramification theory, and the geometry of linear systems.

The constructions introduced here suggest several possible directions for further
investigation, including a deeper study of the singularities of higher Scorza
correspondences, their behavior in moduli, and the geometry of analogous loci
associated with more general theta characteristics or higher rank vector bundles.

A natural open problem in the study of Scorza correspondences is whether they satisfy a Torelli-type property. Namely, one may ask whether a general even spin curve $(C,\eta)$ can be reconstructed uniquely from its associated Scorza curve $\Gamma_\eta$.
Equivalently, one can ask whether the rational map
$$
\begin{array}{rcl}
\sigma: &\overline{S}_g^+ & \dashrightarrow  \overline{\mathcal {M}}_{1+3g(g-1)} \\
    & (C,\eta) & \longmapsto  \Gamma_\eta
\end{array}
$$
is generically injective.

\subsection*{Acknowledgments} 
The author wishes to thank her advisor Gian Pietro Pirola for his constant support, guidance, and encouragement.
She is grateful to Angela Ortega and Gavril Farkas for their hospitality and for many stimulating discussions during her stay in Berlin.
The author also thanks Daniele Agostini, Davide Bricalli and Irene Spelta for valuable discussions and helpful advice related to this work.

\section{Scorza correspondence}
We first collect some basis facts and fix some notations.
\subsection{Preliminaries on line bundles and correspondences}
Let $C$ be a smooth complex projective curve of genus $g\ge 2$. 

We denote by $\Div(C)$ the group of Weil divisors on $C$ and by $\Div^d(C)$ the subgroup of divisors of degree $d$. Since $C$ is smooth, Weil and Cartier divisors coincide, so to every divisor $D$ is associated an invertible sheaf, namely a line bundle, $\mathcal{O}_C(D)$. 

We recall that the Picard group of $C$, $\Pic(C)$, is defined as the group of isomorphism classes of invertible sheaves on $C$, with structure given by the tensor product. It can be expressed as 
$$
\Pic(C)=H^1(C,\mathcal{O}_C^*)=\bigsqcup_{d\in\mathbb Z}\Pic^d(C),
$$
where
$
\Pic^d(C):=\{[D]\in \Pic(C)\mid \deg D=d\}$ parametrizes isomorphism classes of line bundles of degree $d$ on $C$.

In particular, $\Pic^0(C)$ is canonically isomorphic to the Jacobian variety $J(C)$ of $C$, and it is an Abelian variety of dimension $g$.

We denote by
$$
W_{g-1}:=\{L\in \Pic^{g-1}(C)\mid h^0(C,L)\ge 1\}
$$
the locus of effective divisor classes of degree $g-1$. This is a divisor in $\Pic^{g-1}(C)$, called the theta divisor of the curve.
\\Fixing a divisor $\mathcal A\in \Pic^{g-1}(C)$, we can also define the associated theta divisor on the Jacobian $J(C)$ by
$$
\Theta_{\mathcal A}:=W_{g-1}-\mathcal A.
$$
Equivalently,
$$
\Theta_{\mathcal A}=\{M\in Pic^0(C)\mid h^0(C,M\otimes \mathcal A)\ge 1\}.
$$
We will consider the following maps:

\begin{itemize}
    \item the projections
    $$
    \pi_i:C\times C\longrightarrow C,\qquad i=1,2,
    $$
    given by
    $$
    \pi_1(p,q)=p,\qquad \pi_2(p,q)=q;
    $$

    \item the involution exchanging the two factors
    $$
    \iota:C\times C\longrightarrow C\times C,
    $$
    defined by
    $$
    \iota(p,q)=(q,p);
    $$

    \item the difference map
    $$
    u:C\times C\longrightarrow J(C),
    $$
    defined by
    $$
    u(p,q)=\mathcal O_C(p-q).
    $$
\end{itemize}
Let $\mathcal A$ be a line bundle on $C$ of degree $\deg(\mathcal A)=g-1$ such that $h^0(\mathcal{A}):=\dim H^0(C,\mathcal{A})=0$. We define
$$
\Gamma_{\mathcal A}:=\{(p,q)\in C\times C\mid h^0(C,\mathcal A\otimes\mathcal O_C(p-q))>0\}\subset C\times C.
$$
Equivalently, using the theta divisor $\Theta_{\mathcal A}\subset J(C)$ defined above, we have set–theoretically
$$
\Gamma_{\mathcal A}=u^{-1}(\Theta_{\mathcal A}).
$$
Since $\deg(\mathcal A)=g-1$, for every $p\in C$ we have $\deg(\mathcal A\otimes\mathcal O_C(p))=g$ and hence, by Riemann–Roch and the assumption $h^0(C,\mathcal A)=0$, one obtains
$$
h^0(C,\mathcal A\otimes\mathcal O_C(p))=1.
$$
It follows that for every $p\in C$ there are exactly $g$ points $q_1,\dots,q_g$ (eventually with repetitions) of  $C$ such that $(p,q_i)\in\Gamma_{\mathcal A}$ for all  $i=1,\dots,g$. Therefore $\Gamma_{\mathcal A}$ defines a correspondence on $C\times C$ of degree $g$.
Moreover $\Gamma_{\mathcal A}$ is a curve in $C\times C$ 
 Indeed, if $(p,p)\in\Gamma_{\mathcal A}$ then $h^0(C,\mathcal A)>0$, contradicting the assumption on $\mathcal A$. Hence
$$
\Gamma_{\mathcal A}\cap\Delta=\varnothing,
$$
where $\Delta\subset C\times C$ denotes the diagonal. So the inverse image of $\Gamma_{\mathcal{A}}$ under $u$ cannot be the whole surface. Since $\Theta_{\mathcal A}$ is an ample divisor on $J(C)$, then $u^{-1}(\Theta_{\mathcal A})$ cannot be trivial and therefore $\Gamma_{\mathcal A}$ is a proper one dimensional subvariety of $C\times C$.
\begin{remark} Let us note that the case $h^0(\mathcal{A})=0$ is the most interesting in a geometric point of view. Indeed, if $h^0(\mathcal{A})\geq 2$ then the locus defined above coincides with $C\times C$. If instead $h^0(\mathcal{A})=1$ and we denote by $D$ the unique effective divisor in $\mathcal{A}$, then the correspondence is, set theoretically, equal to: 
$$
\Delta \;\cup \bigcup_{p \in \operatorname{Supp}(K_C-D)} \{p\} \times C 
\;\cup \bigcup_{q \in \operatorname{Supp}(D)} C \times \{q\}.
$$
\end{remark}
We define
$$
\Gamma'_{\mathcal A}:=\iota(\Gamma_{\mathcal A}).
$$
Using Riemann–Roch, one easily checks that the condition defining $\Gamma'_{\mathcal A}$ is equivalent to
$$
h^0\big(C,(K_C-\mathcal A)\otimes\mathcal O_C(p-q)\big)>0,
$$
hence
$$
\Gamma'_{\mathcal A}=\Gamma_{K_C-\mathcal A}.
$$
\\We will use the usual the notation over $C\times C$
$$
L\boxtimes M:=\pi_1^*L\otimes\pi_2^*M
$$
for line bundles $L,M$ on $C$. 

For a fixed $p\in C$, restriction to the vertical fiber yields
$$
(\Gamma_{\mathcal A}){|_{{p}\times C}}\cong u_p^{-1}(\Theta_{\mathcal A}),
\qquad
u_p(q)=\mathcal O_C(p-q).
$$
By definition of $\Theta_{\mathcal A}$, this divisor is the unique effective representative of the linear system $|p+\mathcal A|$. Let $\mathcal L:=\mathcal O_{C\times C}(\Gamma_{\mathcal A})$, hence
$$
(\Gamma_{\mathcal A}){|_{{p}\times C}}\sim p+\mathcal A,
\qquad
\mathcal L{|_{{p}\times C}}\cong\mathcal O_C(p+\mathcal A).
$$
Similarly, for a fixed $q\in C$, restriction to the horizontal fiber gives
$$
(\Gamma_{\mathcal A}){|_{C\times{q}}}\cong u_q^{-1}(\Theta_{\mathcal A}),
\qquad
u_q(p)=\mathcal O_C(p-q).
$$
By what we assumed below one obtains
$$
(\Gamma_{\mathcal A}){|_{C\times{q}}}\sim q+(K_C-\mathcal A),
\qquad
\mathcal L{|_{C\times{q}}}\cong\mathcal O_C(q+(K_C-\mathcal A)).
$$
By the seesaw theorem, these two restrictions uniquely determine the line bundle $\mathcal L$ and therefore
\begin{equation} \label{line bundle}
\mathcal O_{C\times C}(\Gamma_{\mathcal A})
\ \cong
\mathcal O_{C\times C}(\Delta)\otimes
\big((K_C-\mathcal A)\boxtimes\mathcal A\big).
\end{equation}
If $\mathcal{A}$ is a theta characteristic, say $\eta$, then one has $K_C - \eta \sim \eta$. 

In particular, using \eqref{line bundle}, one can show that if $\mathcal{A}\in \Pic^{g-1}(C)$with $h^0(\mathcal{A})=0$ then $\Gamma_{\mathcal{A}}$ is symmetric under the action of $\iota$ if and only if $\mathcal{A}$ is a theta characteristic. 

\subsection{Scorza correspondence associated to spin structure}
Let ${S}_g$ be the moduli space of spin curves of genus $g$. Depending on the parity of the spin structure, we distinguish two components ${S}_g^+$ and ${S}_g^-$, whose geometry has been studied by \cite{Cornalba1989}, \cite{Farkas2010}, \cite{FarkasVerra2014}.

We rephrase the above situation as:
\begin{definition}
Let $(C,\eta)\in {S}_g^+$ an even spin curve with $h^0(\eta)=0$. Then we define
\[
\Gamma_{\eta} := \{\, (p,q) \in C \times C \mid h^0(C,\, \eta \otimes \mathcal{O}_C(p-q)) > 0 \,\} \ \subset C \times C,
\]
which is called the \emph{Scorza correspondence}. 
\end{definition} 
This correspondence has been classically introduced by Gaetano Scorza in \cite{Scorza1899}, \cite{Scorza1900} and it is a curve in $C \times C$ of genus 
$g(\Gamma_{\eta}) = 3g(g-1)+1$ [see \cite{DolgachevKanev1993}]. 
Assuming that $\Gamma_{\eta}$ is reduced, its numerical class in the Néron–Severi group of $C\times C$ is computed in \cite{DolgachevKanev1993} and is given by
$$
\Gamma_{\eta} \equiv (g-1)C_1 + (g-1)C_2 + \Delta,
$$
where $C_i \in H^2(C\times C,\Qset )$ denotes the fiber class on the two projections on each factor of $C\times C$.

The involution $\iota: \Gamma_{\eta} \to \Gamma_{\eta}$ restricted to $\Gamma_{\eta}$ is fixed point free, since $\Gamma_{\eta}$ does not intersect the diagonal.

It is easy to show that the Scorza curve $\Gamma_{\eta}$ is big and nef, $1-$connected, and so is connected.
\\We exclude the case $g = 1$ from our discussion, as in this situation the Scorza curve $\Gamma_{\eta}$ is not of interest from a geometric perspective since it is simply the graph of the translation by $\eta$ automorphism of $C$, and hence is isomorphic to $C$ itself.

In \cite{FarkasIzadi2024}, the authors extend the definition of the Scorza curve, describing the limiting behavior in the boundary of the moduli space, i.e., as a point in \(\overline{\mathcal{M}}_{1 + 3g(g-1)}\), for a general spin curve \((C, \eta) \in \Theta_{\mathrm{null}}=\{(C,\eta)\in S_g^+ \vert h^0(C,\eta)>0 \} \).

To this end, let $(C,\eta)\in\Theta_{\mathrm{null}}$ be a general spin curve with $h^0(C,\eta)=2$, and let $f:C\to\mathbb{P}^1$ be the degree $g-1$ morphism induced by the pencil $|\eta|$, which has only simple ramification points. Denote by $x_1,\dots,x_{4g-4}$ the ramification points of $f$. The authors construct the double cover $\tilde{C}_\eta \to C$ associated to the canonical bundle $\omega_C$, branched precisely over the divisor $x_1 + \cdots + x_{4g-4}$.

They define the \emph{trace curve}
\[
\Gamma_\eta^{(2)} := \{ x+y \in C^{(2)} \mid H^0(C, \eta(-x - y)) \ne 0 \}
\]
where $C^{(2)}$ is the symmetric product, and consider its preimage in $C\times C$ via the quotient map $q: C \times C \to C^{(2)}$, that is,
\[
\tilde{\Gamma}_\eta := q^{-1}(\Gamma_\eta^{(2)}) \subset C \times C.
\]
Both $\Gamma_\eta^{(2)}$ and $\tilde{\Gamma}_\eta$ are smooth and reduced. However, the initial candidate for the limiting Scorza curve in $C \times C$ is the non-reduced and non-stable curve $\tilde{\Gamma}_\eta + 2\Delta$.

After performing the appropriate stabilization procedure, the authors prove the following:

\begin{theorem}[G.~Farkas, E.~Izadi {\cite{FarkasIzadi2024}}]
	For a general even spin curve $(C,\eta)\in \Theta_{\mathrm{null}}$, the limiting Scorza correspondence $\Gamma_{\eta}$ is the transverse union of $\tilde\Gamma_\eta$ and the curve $\tilde{C}_\eta$, intersecting at the $4g - 4$ diagonal points $(x_1, x_1), \dots, (x_{4g - 4}, x_{4g - 4})$.
\end{theorem}
Moreover, from \cite{FarkasVerra2014} we have the following: 
\begin{theorem}[G.~Farkas, A.~Verra] \label{liscezza}
	For a general theta-characteristic $(C,\eta)\in S_g^+$, the Scorza curve $\Gamma_{\eta}$ is a smooth curve.
\end{theorem}
The authors prove smoothness of $\Gamma_{\eta}$ by induction on $g$: assuming a singular point exists, they degenerate $C$ to a nodal curve and analyze the resulting limit linear series, reaching a contradiction in all possible cases via incompatible vanishing conditions on the boundary.

We now present an alternative proof of the same statement, employing a different approach that combines the explicit description of divisors via equations with the study of the Griffiths infinitesimal invariant, referring to \cite{Griffiths1983}. In this sense, we will need the following criterion from \cite{FassinaPirola2026}:
\begin{proposition}\label{prop fanzo}
    Let $[C]$ be a general point in $\mathcal{M}_g$ with $g \geq 2$. Let $A, B$ be two distinct divisors of $C$. Assume there exists an integer $n$ such that $nA$ is linearly equivalent to $nB$. Assume that the support $\mathrm{supp}(A-B) = \{p_1, \dots, p_d\}$. Then
    \[
        h^0(K_C - \sum_{i=1}^d  p_i) = 0.
    \]
\end{proposition}
\noindent Before proceeding with the proof of Theorem 1.3, we recall a useful smoothness criterion for subvarieties of product spaces. This criterion can be viewed as a straightforward application of standard linear algebra applied to projections maps and tangent spaces.

\begin{proposition}\label{criteriogriffiths}
Let $C$ be a smooth complex projective curve, and let $D \subset C^n$ be a subvariety of dimension $k$. For any subset of indices $I \subset \{1, \dots, n\}$ with $|I| = k$, let 
\[
\pi_I: C^n \longrightarrow C^k
\]
denote the natural projection onto the $k$ factors indexed by $I$. Then, a point $p \in D$ is smooth if and only if there exists a choice of $I$ such that the differential of the restriction
\[
d(\pi_I|_D)_p: T_p D \longrightarrow T_{\pi_I(p)} C^k
\]
is an isomorphism.
\end{proposition}

\begin{remark}\label{rem:coincident_points}
In the special case where $n=2$ and $k=1$, the variety $D \subset C \times C$ defines an algebraic correspondence $T$ of degree $g$. Under these assumptions, the infinitesimal condition stated in Proposition \ref{criteriogriffiths} admits a classical geometric interpretation in terms of ramification. Specifically, the failure of the differential $d(\pi_I|_D)_p$ to be an isomorphism means that the projection maps ramify at $p$. 

According to the classical criterion of Griffiths--Harris \cite[p.283]{GriffithsHarris1978}, a point $p=(x,y) \in D$ is a singular point of the correspondence curve $D$ only if it is a coincident point of $T$, meaning that $y$ appears in the image $T(x)$ with multiplicity greater than or equal to $2$ (and symmetrically for the second projection if the correspondence is symmetric). Thus, Proposition \ref{criteriogriffiths} generalizes this standard fact to higher dimensions and higher codimensions.
\end{remark}
We can now present a different proof of \ref{liscezza}, assuming that $(C,\eta)$ is a general point in $S^+_g$ with $g \geq 2$:
\begin{proof}[Proof (Theorem~\ref{liscezza})]
For a general pair $(C,\eta)$, fixing a point $p \in C$, the intersection of the Scorza curve $\Gamma_{\eta}$ with the fiber over $p$ consists of $g$ points (and symmetrically for the fiber over $q$). These are precisely the points $(p,q_1), \dots, (p,q_g)$, where the divisor associated to $\eta \otimes \mathcal{O}_C(p)$ is linearly equivalent to $q_1 + \dots + q_g$. If the points $(p,q_i) \in \Gamma_{\eta}$ are distinct, then by \ref{criteriogriffiths} they are smooth points of the Scorza curve, which implies that $\Gamma_{\eta}$ is smooth in a neighborhood of these points.

This observation leads to a characterization of the singular locus of $\Gamma_{\eta}$: if a point $(p,q) \in \Gamma_{\eta}$ is singular then
\[
h^0\big(C, \eta \otimes \mathcal{O}_C(p - 2q)\big) > 0 
\quad \text{and} \quad 
h^0\big(C, \eta \otimes \mathcal{O}_C(q - 2p)\big) > 0.
\]
Equivalently, these conditions mean that there exist unique effective divisors $D$ and $E$ such that
\begin{equation} \label{singolarità}
\begin{cases} 
	\eta + p - 2q \sim D, \\
	\eta + q - 2p \sim E.
\end{cases}
\end{equation}
One observes that $D + 2q = \eta + p$, so that $\operatorname{supp}(D + 2q)$ consists of the $g$ points (counted with multiplicity) corresponding to $p$ on $\Gamma_{\eta}$. In particular, since $\Gamma_{\eta} \cap \Delta = \varnothing$, it follows that $p \notin \operatorname{supp}(D)$. Analogously, one has $q \notin \operatorname{supp}(E)$.
\\Now, adding $p$ to the first equation and $q$ to the second one in \eqref{singolarità}, we obtain:
$$
\begin{cases}
	\eta +2p - 2q \sim D +p, \\
	\eta +2q - 2p \sim E +q.
\end{cases}
$$
Subtracting the two conditions, we deduce:
$$
E +3p \sim D +3q.
$$
We claim that $E+3p$ and $D+3q$ are distinct divisors. Indeed, suppose by contradiction that
$$
E+3p=D+3q.
$$
Since $p\neq q$, comparing the two sides shows that $p\in \supp(D)$ and $q\in \supp(E)$. But this contradicts the previous observation that $p\notin \operatorname{supp}(D)$ and $q\notin \operatorname{supp}(E)$. 
\\Therefore, $E+3p$ and $D+3q$ are distinct linearly equivalent divisors. We can then apply Proposition~\ref{prop fanzo} to $A = E+3p$ and $B = D+3q$ (with $n=1$), which yields
\begin{equation}\label{fanzo2}
h^0(K_C - \supp(A-B)) = 0.
\end{equation}
On the other hand, by adding the two equations in \eqref{singolarità}, we get
$$
2\eta - p - q \sim D + E.
$$
Since $\eta$ is a theta characteristic ($2\eta \sim K_C$), we have
$$
K_C \sim D + E + p + q
$$
which gives that the divisor $K_C-\supp(A-B)$ is effective, contradicting \eqref{fanzo2}. This concludes the proof.
\end{proof}

\subsection{Base Points and self-composition}

Let $(C,\eta)$ be an even ineffective spin curve of genus $g\geq 2$. 
In this subsection we study the base locus of the linear system
$$
\vert \eta\otimes\mathcal{O}_C(p+q)\vert,
$$
for varying pairs $(p,q)\in C\times C$. 
As we shall see, this problem is naturally related to the geometry of the Scorza correspondence and, more precisely, to itself composition.
\\The results established in this section will be used in Section~3, in particular for the analysis of the smoothness of the Scorza threefold.

First, a point $r \in C$ is a base point of $\vert \eta\otimes\mathcal{O}_C(p+q)\vert$ if and only if
\begin{equation} \label{condizione punti base}
    h^0\big(\eta\otimes \mathcal{O}_C(p+q-r)\big)=2.
\end{equation}
We are therefore led to consider the locus of triples $(p,q,r)\in C^3$ satisfying \eqref{condizione punti base}.
\\As we shall see, this condition gives rise to a natural curve in $C\times C$, which admits a simple description in terms of the Scorza correspondence.

\begin{lemma} \label{base}
Let $(C,\eta)$ be an ineffective even spin curve. A point $r\in C$ is a base point of $\vert \eta\otimes\mathcal{O}_C(p+q)\vert$ if and only if there exist effective divisors $D,E,F$ on $C$ such that
\begin{equation} \label{divisori}
\begin{cases}
\eta + p \sim D + r \\
\eta + q \sim E + r,\\
\eta + r \sim F + p + q. 
\end{cases}
\end{equation}
\end{lemma}

\begin{proof}
Assume that
$
h^0(\eta\otimes\mathcal{O}_C(p+q-r))=2.
$ Let
$
s_1\in H^0(\eta\otimes\mathcal{O}_C(p))$ and $
s_2\in H^0(\eta\otimes\mathcal{O}_C(q))
$
be generators of the corresponding one--dimensional spaces. Since the linear systems
$
\vert\eta\otimes\mathcal{O}_C(p)\vert$ and $
\vert\eta\otimes\mathcal{O}_C(q)\vert
$
are strictly contained in
$
\vert\eta\otimes\mathcal{O}_C(p+q)\vert,
$
the section $s_1$ vanishes at $q$, while $s_2$ vanishes at $p$. Therefore both may be regarded as sections of
$
H^0(\eta\otimes\mathcal{O}_C(p+q-r)),
$ by our hypothesis. 
Moreover, they are linearly independent. If not, $s_1$ would vanish also at $q$ descending to a non zero section of $\eta$, yielding a contradiction since $\eta$ is ineffective. Since
$
h^0(\eta\otimes\mathcal{O}_C(p+q-r))=2,
$
the sections $s_1,s_2$ form a basis.
So, from the section $s_2$ we obtain an effective divisor $E$ such that
$$
E + p \sim \eta +p+q-r,
$$
which is equivalent to
$$
\eta + q \sim E + r.
$$
Similarly, from $s_1$ we obtain an effective divisor $D$ such that
$$
\eta + p \sim D + r.
$$
Finally, from the obtained conditions and the simmetry of the Scorza construction, we deduce that both $p$ and $q$ lie in the divisor of a section of $\eta\otimes \mathcal{O}_C(r)$, and so
$$
\eta + r \sim F + p+q 
$$
for some effective divisor $F$.

Conversely, one can easily see that if there exist effective divisors $D,E,F$ as in $\eqref{divisori}$, these linear equivalences produce two independent sections of 
$
\vert\eta\otimes \mathcal{O}_C(p+q-r)\vert,
$
hence
$
h^0(\eta\otimes \mathcal{O}_C(p+q-r))=2.
$
\end{proof}
From the last equality in the previous result, one can deduce the following:
\begin{proposition} \label{curve}
Let $B\subset C\times C$ be the locus parametrizing those pairs $(p,q)$ for which there exists a point $r\in C$ satisfying
$
h^0(\eta\otimes \mathcal{O}_C(p+q-r))=2.
$ Then $B$ is a curve in $C\times C$. 

Moreover, if $(C,\eta)$ is general, then $B$ does not intersect the diagonal.
\end{proposition}
\begin{proof}
Let us consider the incidence correspondence
$$
Z=\{(p,q,r)\in C\times C\times C \mid
h^0(\eta\otimes\mathcal O_C(p+q-r))=2\}\subseteq C\times C.
$$
Then $B=\pi_{12}(Z)$,
where $\pi_{12}:C\times C\times C\to C\times C$ is the projection onto the first two factors.

Let us look at the third projection
$\pi_3:Z\to C$
and, in particular, at the fibre $(\pi_3|_Z)^{-1}(r)$ over a point $r\in C$. For any point $(p,q)\in (\pi_3|_Z)^{-1}(r)$ we know by Lemma~\ref{base} that there exists an
effective divisor $F$ such that
$F+p+q\sim \eta + r$.

We know by Riemann--Roch that $h^0(\eta(r))=1$ and hence there exists a unique effective divisor in $|\eta(r)|$, namely $F+p+q$, which has degree $g$. There there can exist at most finitely many pairs $(p,q)$ in the fiber $(\pi_3|_Z)^{-1}(r)$. Therefore $ \dim Z=1 $
and hence $\dim B\leq 1$.

We now rule out the case where $\dim B=0$. Let us assume by contradiction that $\dim B=0$.
If $(p,q)\in B$, then $(\pi_{12}|_Z)^{-1}(p,q)$
has dimension $1$, i.e. $(\pi_{12}|_Z)^{-1}(p,q)=C$,
which means that for every $r\in C$ one has $ h^0(\eta\otimes\mathcal O_C(p+q-r))=2$. But for $r=p$ we get $
h^0(\eta\otimes\mathcal O_C(q))=2$
contradicting the fact that $
h^0(\eta\otimes \mathcal{O}_C(q))=1$.
\\Therefore $\dim B=1$.
\\Let us now take $C$ general and assume by contradiction that
there exists $(p,p)\in B$. This means that there exists $r\in C$ such that
$$
h^0(\eta(2p-r))=2.
$$
By Riemann--Roch we get also
$$
h^0(\eta(r-2p))=1>0 .
$$
Since $(r,p)\in\Gamma_\eta$ by \ref{base}, from the above condition we get that
$(r,p)$ is singular for $\Gamma_\eta$, which is not possible for $C$
general (see Theorem~\ref{liscezza}). Hence $B$ does not intersect the diagonal $\Delta$.
\end{proof}
We now reinterpret this locus in terms of correspondences on $C$. First, let us give the following, as in \cite{HowardSommese1983}:
\begin{definition} \label{def:composition}
    Given two correspondences $A$ and $B$ in $C\times C$, one defines a product in $C\times C$ as follows: $$ A \circ B=\pi_{13}\big( (A\times X)\cdot (X\times B)\big). $$
\end{definition}
Now, recall that the Scorza correspondence $
\Gamma_\eta \subset C\times C
$
is defined by
$$
(p,r)\in \Gamma_\eta
\quad \Longleftrightarrow \quad
h^0(\eta\otimes \mathcal{O}_C(p-r))>0.
$$
\begin{proposition} \label{baselocus}
Let $(C,\eta)$ be a general element in $ S^+_g$ and let $B\subset C\times C$ be the curve introduced above. Then
$$
B = \Gamma_\eta \circ \Gamma_\eta - g\Delta,
$$
where $\Delta\subset C\times C$ is the diagonal. In particular, its class in the Néron--Severi group is given by:
$$
B\equiv g(g-1)(C_1+C_2)
$$
where $C_i$ is the fiber via the $i-$th projection.
\end{proposition}
\begin{proof}
Assume first that $r$ is a base point as above, i.e. there exists $(p,q) \in B \subset C \times C$
such that $h^0(\eta \otimes \mathcal{O}_C(p+q-r)) = 2$.
By Lemma~\ref{base} we know that (because of the existence of the effective divisors named $D$ and $E$):
$$
h^0(\eta \otimes \mathcal{O}_C(p-r)) > 0 \quad \text{and} \quad h^0(\eta \otimes \mathcal{O}_C(q-r)) > 0
$$
so $(p,r) \in \Gamma_\eta$ and, by symmetry, $(r,q) \in \Gamma_\eta$. This implies, by definition, that $(p,q) \in \Gamma_\eta \circ \Gamma_\eta.$

Therefore, we have shown $B \subset \Gamma_\eta \cdot \Gamma_\eta$ (set-thereotically). As observed at the beginning, for a point $p \in C$ the fiber of $\Gamma_\eta$ over $p$ consists of $g$ points $q_1, \ldots, q_g$ with multiplicity, so for each of these points we have that $(p, p) \in \Gamma_\eta \circ \Gamma_\eta.$
We can then write $$\Gamma_\eta \circ \Gamma_\eta = S + g\Delta$$ for some residual cycle $S$.

We claim that $S = B$. First we show that $S \subseteq B$. To this end let us take a point $(p,q) \in \Gamma_\eta \circ \Gamma_\eta$ with $p \neq q$.
Then there exists $r \in C$ such that $ h^0(\eta \otimes \mathcal{O}_C(p-r)) > 0$,
and by symmetry $h^0(\eta \otimes \mathcal{O}_C(q-r)) > 0$.
\\Let us take two non-zero sections $\sigma_1 \in H^0(\eta \otimes \mathcal{O}_C(p-r))$ and
$\sigma_2 \in H^0(\eta \otimes \mathcal{O}_C(q-r))$.
Multiplying these by the canonical sections of $\mathcal{O}_C(q)$ and $\mathcal{O}_C(p)$ we get
two non-zero sections $s_1, s_2 \in H^0(\eta \otimes \mathcal{O}_C(p+q-r))$ such that \ $s_1$ vanishes at $q$
and $s_2$ vanishes at $p$.
\\These are independent: if not, one can see that $\sigma_1$ vanishes at $p$ and so it would be a
section of $H^0(\eta(-r))$ and so of $H^0(\eta)$, which is not possible, since $\eta$ is ineffective. So $h^0(\eta \otimes \mathcal{O}_C(p+q-r)) \geq 2$.

Since $H^0(\eta \otimes \mathcal{O}_C(p+q-r)) \hookrightarrow H^0(\eta \otimes \mathcal{O}_C(p+q))$ and by Riemann--Roch $h^0(\eta \otimes \mathcal{O}_C(p+q))=2$, we have that $(p,q) \in B$.

Now, let us show that $B \subseteq S$: this is clear, since we have shown that $B \subseteq \Gamma_\eta \cdot \Gamma_\eta$
and, from Proposition~\ref{curve}, one has that $B$ does not intersect $\Delta$.
\medskip
Finally, it is known by \cite{HowardSommese1983} that
$\Gamma_\eta \circ \Gamma_\eta = g\Gamma_\eta \equiv g(g-1)(C_1 + C_2) + g\Delta$,
so subtracting the diagonal component we get
\[
B \equiv g(g-1)(C_1 + C_2).
\]
\end{proof}
\section{The \protect\MakeLowercase{m}-Scorza correspondence}
The purpose of this section is to generalize the classical Scorza correspondence to the case where the twist of the theta characteristic $\eta$ is of higher order along the diagonal.
\\More precisely, 
\begin{definition} \label{nScorza} For a smooth curve $ (C,\eta)\in S^+_g$ with $h^0(\eta)=0$ and an integer $m > 1 $, we define the \emph{\(m\)-Scorza correspondence} as
$$
\Gamma^m_{\eta} := \left\{ (p,q) \in C \times C \,\middle|\, h^0\big( \eta \otimes \mathcal{O}_C(mp - mq) \big) > 0 \right\} \subset C \times C.
$$
\end{definition} 
This curve generalizes the classical Scorza curve (which corresponds to \( m = 1 \)) and inherits several of its properties. One can see, for example, that it is symmetric and it does not intersect the diagonal.
\begin{proposition}
Let $(C,\eta)$ be a very general point of $S^+_g$. Then the class of $\Gamma^m_{\eta}$ in the Néron--Severi group of $C\times C$ is 
$$
\Gamma^m_{\eta} \equiv m^2 \cdot \left[ (g - 1)C_1 + (g - 1)C_2 + \Delta \right]
$$
where, again, $C_i$ are the fibers over the $i-$th projection and $\Delta$ denotes the diagonal, and 
$$p_a(\Gamma^m_{\eta}) = (2m^2 + m^4)g(g - 1) + 1.$$
\end{proposition}
\begin{proof}
For a very general curve $C$, one has
$\operatorname{End}(\operatorname{Jac}(C))=\mathbb Z$
(see \cite{Koizumi1976}),
and consequently
$$
\operatorname{NS}(C\times C)
=
\mathbb Z\langle C_1,C_2,\Delta\rangle .
$$
Hence the class of $\Gamma^m_{\eta}$ in $\mathrm{NS}(C\times C)$ must be of the form
$$
\Gamma^m_{\eta} \equiv a(m)C_1 + b(m)C_2 + c(m)\Delta.
$$
By Riemann--Roch, the correspondence $\Gamma^m_{\eta}$ is symmetric and disjoint from the diagonal $\Delta$. In particular, the simmetry implies that  $a(m)=b(m)$ and by $\Gamma^m_{\eta}\cdot \Delta=0$ we find $a(m)=(g-1)c(m)$. So its class in the Néron–Severi group of $C\times C$ must be a scalar multiple of the class of the classical Scorza correspondence, i.e.:
$$
\Gamma^m_{\eta} \equiv c(m)\big[(g-1)C_1 + (g-1)C_2 + \Delta\big],
$$
for some integer $c(m)$.
\\To determine the value of $c(m)$, we compute the intersection number of $ \Gamma^m_{\eta} $ with a fiber, for example $C_1 $. This intersection is equal to $ c(m)g $. On the other hand, we compute it directly as follows.

Since \( \deg \eta = g - 1 \), the twist \( \eta \otimes \mathcal{O}_C(mp) \) has degree \( g + m - 1 \), and by Riemann–Roch, for each \( m > 1 \), we have
$$
h^0\big( \eta \otimes \mathcal{O}_C(mp) \big) = m.
$$
Thus, the linear system \( |\eta \otimes \mathcal{O}_C(mp)| \) induces a morphism
$$
f : C \longrightarrow \mathbb{P}^{m-1}
$$
of degree \( g + m - 1 \). The intersection of $\Gamma^m_{\eta}$ with the fiber \( C_1 \) consists of the \emph{inflection points} of this map, that is, the points at which the ramification of \( f \) is nontrivial. These are precisely the zeros of the Wronskian associated to the linear system \( |\eta \otimes \mathcal{O}_C(mp)| \).

According to the Plücker formula, for a linear system $ (\mathcal{L}, V) $ of dimension $r$ and degree $d$ on a smooth curve $C$, the total ramification is given by the degree of the line bundle \( \mathcal{L}^{r+1} \otimes K_C^{\binom{r+1}{2}} \), that is:
$$
\deg\big( \mathcal{L}^{r+1} \otimes K_C^{\binom{r+1}{2}} \big) = (r+1)\cdot d + (2g - 2)\cdot \binom{r+1}{2} = (r+1)\cdot (d + r(g - 1)).
$$
In our case, the linear system has type \( \mathfrak{g}^{m-1}_{g+m-1} \), so \( r = m - 1 \), \( d = g + m - 1 \), and the total ramification index becomes:
$$
(m-1+1)\cdot\big((g + m - 1) + (m - 1)(g - 1)\big) = m \cdot \big( g + m - 1 + mg - m - g + 1 \big) = m^2 g.
$$
Thus, we conclude that
$$
(\Gamma^m_{\eta} \cdot C_1) = m^2 g = c(m)g,
$$
so that
$c(m) = m^2$, and therefore the class of the \( m \)-Scorza curve is given by
$$
\Gamma^m_{\eta} \equiv m^2 \cdot \left[ (g - 1)C_1 + (g - 1)C_2 + \Delta \right].
$$
Furthermore, applying the adjunction formula, we obtain that the canonical bundle of the \( m \)-Scorza curve satisfies
$$
K_{\Gamma^m_{\eta}} \cong \left( K_{C \times C} \otimes \mathcal{O}_{C \times C}(\Gamma^m_{\eta}) \right)\big|_{\Gamma^m_{\eta}}.
$$
This allows us to compute the arithmetic genus of $\Gamma^m_{\eta}$ by evaluating:
$$
2p_a(\Gamma^m_{\eta}) - 2 = \left( K_{C \times C} + \Gamma^m_{\eta} \right) \cdot \Gamma^m_{\eta}.
$$
Recalling that
$$
K_{C \times C} = (2g - 2)(C_1 + C_2)
$$
we find 
$$
p_a(\Gamma^m_{\eta}) = (2m^2 + m^4)g(g - 1) + 1.
$$
as wanted.
\end{proof}
\begin{remark}
We include here a useful observation, suggested to the author by Davide Bricalli, which provides an alternative conceptual interpretation of the $m$-Scorza correspondence.

Let $u:C\times C \to \mathrm{Pic}^0(C)$ be the difference map $u(p,q)=\mathcal O_C(p-q)$, and let $\Theta_\eta \subset \mathrm{Pic}^0(C)$ be the theta divisor associated to $\eta$. Denote by $[m]:\mathrm{Pic}^0(C)\to \mathrm{Pic}^0(C)$ the multiplication-by-$m$ morphism.

Then the $m$-Scorza correspondence can be equivalently described as
$$
\Gamma^m_\eta = u^{-1}\big([m]^{-1}(\Theta_\eta)\big).
$$
In other words, it is the pullback via the difference map of the inverse image of the theta divisor under multiplication by $m$ on the Jacobian. This recovers the definition in terms of $h^0(\eta\otimes \mathcal O_C(mp-mq))$ and clarifies the role of the $m^2$-scaling in the numerical class of $\Gamma^m_\eta$.
\end{remark}
By construction, for fixed $p\in C$, the vertical fiber
\[
(\Gamma^m_{\eta}) \cap (\{p\}\times C)
\]
is the divisor associated with the linear system $|\eta\otimes\mathcal{O}_C(mp)|$.

This determines the restriction of $\mathcal{O}_{C\times C}(\Gamma^m_{\eta})$ to the fibers of the two projections, and hence its class in $\mathrm{NS}(C\times C)$. In particular, there exists a line bundle $L_m$ on $C$ such that
\[
\mathcal{O}_{C\times C}(\Gamma^m_{\eta})
\;\equiv\;
\mathcal{O}_{C\times C}(m^2\Delta)
\otimes (L_m \boxtimes L_m),
\]
where $\equiv$ denotes numerical (equivalently, linear) equivalence, since $C$ is general.
\\The line bundle $L_m$ is computed from the linear system $|\eta\otimes\mathcal{O}_C(mp)|$ as
$
L_m \cong \eta^m \otimes K_C^{\binom{m}{2}}.
$
\\Therefore, 
$$
\begin{aligned}
\mathcal{O}_{C\times C}(\Gamma^m_{\eta})
&\;\equiv\;
\mathcal{O}_{C\times C}(m^2\Delta)
\otimes
\big(\eta^m \otimes K_C^{\binom{m}{2}}\big)
\boxtimes
\big(\eta^m \otimes K_C^{\binom{m}{2}}\big)
\\
&:=
\mathcal{O}_{C\times C}(m^2\Delta)
\otimes
\pi_1^*\big(\eta^m \otimes K_C^{\binom{m}{2}}\big)
\otimes
\pi_2^*\big(\eta^m \otimes K_C^{\binom{m}{2}}\big).
\end{aligned} $$
Finally, observe that the parity of $m$ affects the simplification of the factor $\eta^m$. Indeed, since $\eta^2 \cong K_C$, we have
$$
\eta^m \cong
\begin{cases}
K_C^{m/2} & \text{if $m$ is even},\\[4pt]
\eta \otimes K_C^{(m-1)/2} & \text{if $m$ is odd}.
\end{cases}
$$
\\As in the case of the classical Scorza curve, the $m$-Scorza curve is connected, as its class in the Neron-Severi group is big and nef. We now aim to adapt the techniques previously employed to study its smoothness, in order to prove the following result:

\begin{theorem} \label{liscezza2}
	For a general element $(C, \eta)\in S^+_g$ with $g\geq 2$, the $m$-Scorza curve $(m\geq 2)$, is smooth away from the intersection locus
	$\Gamma^m_{\eta} \cap \Gamma^{m-1}_{\eta}$.
\end{theorem}

\begin{proof}
Fix a general point $p\in C$ and consider the complete linear system
$
|\eta\otimes\mathcal{O}_C(mp)|.
$
Since
$
h^0(\eta\otimes\mathcal{O}_C(mp))=m,
$
this linear system defines a morphism
$
\varphi_p:C\longrightarrow \mathbb P^{m-1}.
$
The points $q\in C$ such that
$
(p,q)\in \Gamma_\eta^m
$
are precisely the inflection points of the map $\varphi_p$, namely the points satisfying
$
h^0\big(\eta\otimes\mathcal{O}_C(mp-mq)\big)>0.
$
By the Plücker formula, these are in number $m^2g$. Hence there are exactly $m^2g$ inflection points counted with their weights.
\\When all these points are distinct, each of them has weight equal to $1$, hence corresponds to a simple ramification point of the projection on each factor. By Proposition~\ref{criteriogriffiths}, the corresponding points of $\Gamma_\eta^m$ are smooth.

Therefore, possible singularities may only occur at those points for which the vanishing sequence of the linear system
$
|\eta\otimes\mathcal{O}_C(mp)|
$
is non-generic, i.e. is different from $(0,1,\dots,m-1)$ (at least one point has to be of weight greater than $2$). We distinguish two cases.

\medskip

\noindent
{\bf Case 1.} The vanishing sequence is of the form
$$
(0,\dots,m-2,m+k) \quad \text{with $k\geq 1$}.
$$
Equivalently,
\begin{equation} \label{gap generico}
h^0\left( \eta \otimes \mathcal{O}_C(mp - (m+k)q) \right) > 0.
\end{equation}
\\At the level of divisors, the condition in \eqref{gap generico} translates into the existence of effective divisors $D$ and $E$, and two integers $k_1, k_2 \geq 1$ such that:
$$
\begin{cases}
\eta +mp \sim (m+k_1)q + D, \\
\eta +mq \sim (m+k_2)p + E,
\end{cases}
$$
where the second relation follows from the symmetry of the correspondence.

Adding $mq$ to the first relation and $mp$ to the second, we obtain
$$
(2m+k_1)q + D \sim (2m+k_2)p + E .
$$
We want to apply Proposition~\ref{prop fanzo} to these two divisors. So let us first observe that they are different. Indeed, as $p\neq q$, if they are equal than we could write $D=D'+(2m+k_2)p$ with $D'$ effective, leading to $h^0(\eta)>0$ which is false. Setting
$$
\tilde{D} := E-D +(2m+k_2)p - (2m+k_1)q,
$$
we find that $\tilde{D}\neq 0$.

On the other hand, adding the two original relations gives
$$
2\eta \cong K_C \sim k_2p + k_1q+ E + D.
$$
Hence the support of $\tilde D$ is contained in the support of a canonical divisor. By Proposition~\ref{prop fanzo}, this is impossible. Therefore such points cannot occur.
\medskip

\noindent
{\bf Case 2.} An additional jump occurs before the last term, so the vanishing sequence is equal to $(\dots,k,k+2,\dots, m)$. Indeed, knowing that
$$
h^0(\eta(mp)) = m,
$$
it may happen, for instance, that
$$
h^0\big(\eta \otimes \mathcal{O}_C(mp - (m-1)q)\big) > 1,
$$
which yields the existence of two effective divisors $X, Y$ such that
\begin{equation} \label{intersezione}
\begin{cases}
\eta +mp \sim (m-1)q + X, \\
\eta +mp \sim  mq + Y.
\end{cases}
\end{equation}

Observe that in $H^0(\eta\otimes\mathcal{O}_C(mp))$ (of dimension $m$) we have two subspaces, namely $H^0(\eta\otimes\mathcal{O}_C((m-1)p))$ and $H^0((\eta\otimes\mathcal{O}_C(mp-(m-1)q)))$, of dimensions respectively equal to $m-1\geq 1$ and $k\geq 2$. These subspaces have then non-trivial intersections.

This implies that one of the elements of $H^0((\eta\otimes\mathcal{O}_C(mp-(m-1)q)))$ must vanish at $p$, and therefore we may refine the first equation in \eqref{intersezione} as
$$
\eta +mp \sim (m-1)q + p + \tilde{X},
$$
for some effective divisor $\tilde{X}$. Equivalently, this yields
$$
\begin{cases}
\eta +(m-1)p \sim (m-1)q + \tilde{X}, \\
\eta +mp \sim mq + Y,
\end{cases}
$$
which proves that such a point $(p,q)$ belongs to the intersection locus
$$
\Gamma^m_{\eta} \cap \Gamma^{m-1}_{\eta}.
$$


By \ref{criteriogriffiths}, such higher ramification cannot occur at a smooth point of the correspondence.
\\We have then shown that $\Sing(\Gamma_ \eta^m)\subseteq \Gamma_ \eta^m \cap \Gamma_ \eta^{m-1} $ as claimed.

Finally, observe that the two types of non-generic vanishing considered above could a priori occur in different ways on the two projections. However, this introduces no additional cases. Indeed, in Case~1 we use that both $p$ and $q$ occurs with the same ramification type. The only genuinely mixed situation is when one projection is of the type described in Case~1 and the other in Case~2, which is already covered by Case~2 after exchanging the roles of p and q.
\end{proof}
\section{Scorza threefold}
It is natural to consider constructions analogous to the classical Scorza correspondence in higher-dimensional settings. More precisely, given an even ineffective theta-characteristic $\eta$ on $C$, one can associate to it some subvarieties in higher products of the curve defined by imposing vanishing conditions on suitable twists of $\eta$.
\begin{definition}
For any integer $n \geq 1$, let us consider the locus
$$
\Gamma_n = \Big\{ (p_1,\dots,p_n,q_1,\dots,q_n) \in C^{2n} \;\Big|\; 
h^0\big(C,\eta \otimes \mathcal{O}_C(\textstyle{\sum_{i=1}^n} p_i - \sum_{i=1}^n q_i)\big) > 0 \Big\}.
$$
This defines a natural hypersurface in $C^{2n}$, the \textit{Scorza hypersurface}.
\end{definition}
\begin{remark}[Relation with Scorza correspondences on symmetric products]
\label{rem:AKP-symmetric-products}
Agostini, Kummer, and Park~\cite{AKP26} introduce a related
construction on symmetric products $C^{(k+1)}$. For an ineffective theta characteristic
\(\eta\), they define the higher Scorza correspondence
\[
\left\{
(\xi,\xi')\in C^{(k+1)}\times C^{(k+1)}
\ \middle|\
h^0\bigl(C,\eta(\xi'-\xi)\bigr)>0
\right\},
\]
endowed with the scheme structure given by the zero locus of the
\( (k+1)\)-Szeg\H{o} kernel associated with \(\eta\).

Thus, the Scorza hypersurface \(\Gamma_n\) can be viewed as the pullback
to the ordered Cartesian product of the corresponding Scorza divisor on
the product of symmetric powers. In~\cite{AKP26}, this construction is
used to obtain admissible determinantal representations and Ulrich sheaves on higher secant varieties; here, instead, we investigate
the numerical and local geometry of these hypersurfaces.
\end{remark}

In this section, we focus on the first nontrivial case beyond the classical one, namely $n=2$, which leads to a subvariety of the fourfold product $C \times C \times C \times C$.

\begin{definition} 
Let $(C,\eta)\in S^+_g$ be an even ineffective spin curve. We define the Scorza threefold as
$$
\Gamma=\{(p,q,r,s)\in C\times C \times C \times C \mid 
h^0(C,\eta \otimes \mathcal{O}_C(p+q-r-s))>0\}.
$$
\end{definition}
By construction, $\Gamma$ is a threefold inside $C^4$. Indeed, we define the morphism
\[
\delta : C^4 \longrightarrow \Pic^0(C)=J(C), \qquad
(p,q,r,s)\longmapsto \mathcal O_C(p+q-r-s),
\]
that associates to a quadruple of points the corresponding degree zero line bundle.
\\Then
\[
\Gamma := \delta^{-1}(\Theta_\eta) \subset C^4.
\]
Since $\Theta_\eta$ is a divisor on $J(C)$, its pullback via $\delta$ is either empty or a divisor on $C^4$. Moreover, $\Gamma$ is non-empty because $\delta$ is not constant and its image is not contained in $\Theta_\eta$.
\\Finally, observe that the big diagonal
$\{(p,p,p,p)\mid p\in C\} \subset C^4$
is not contained in $\Gamma$, since if $(p,p,p,p)\in \Gamma$ then $h^0(\eta)>0$.
\\It follows that $\Gamma$ is a proper closed subset of $C^4$, and since it is the pullback of a divisor, it has codimension one in $C^4$, hence it is a threefold. Our first goal is to determine its numerical class as a divisor. To do so, for a very general element $(C,\eta)$ in the moduli space $S_{g}^{+}$, let us start by introducing the following notation: let $\pi_i : C^4 \to C$ be the projection onto the $i$-th factor, and set $F_i := \pi_i^{*}[pt]$ for $i = 1, \dots, 4$. For $1 \leq i < j \leq 4$, let $\pi_{ij} : C^4 \to C \times C$ denote the projection onto the $i$-th and $j$-th factors, and define the divisor $\Delta_{ij} := \pi_{ij}^{-1}(\Delta)$, where $\Delta \subset C \times C$ is the diagonal. 

Recall that in the case where $C$ is very general, the Néron--Severi group of $C^4$, has rank $10$: in particular, by using the Künneth decomposition, one sees that it is generated by the four fiber classes $F_i$ together with the six diagonals $\Delta_{ij}$.

\begin{proposition}\label{prop:class_scorza_threefold}
Let $\Gamma \subset C^4$ be the Scorza threefold associated to the very general element $(C,\eta)$ in $S_g^+$. In the Néron--Severi group of $C^4$, the divisor class $\Gamma$ is given by
\[
\Gamma \equiv (g-1)\sum_{i=1}^4 F_i - (\Delta_{12}+\Delta_{34}) + (\Delta_{13}+\Delta_{14}+\Delta_{23}+\Delta_{24}).
\]
\end{proposition}

\begin{proof}
A priori, the class of the Scorza threefold can be written as a linear combination of the generators:
\[
\Gamma\equiv\sum_{k=1}^4\alpha_k F_k+\sum_{1\leq i<j\leq 4}\beta_{ij}\Delta_{ij}.
\]
However, the definition of $\Gamma$ exhibits several symmetries. Applying the Riemann--Roch theorem, we observe that
\[
h^0(C,\eta \otimes \mathcal{O}_C(r+s-p-q))-h^0(C,K_C \otimes \eta^{\vee} \otimes \mathcal{O}_C(p+q-r-s)) = (g-1)-(g-1)=0,
\]
which implies
\[
h^0(C,\eta \otimes \mathcal{O}_C(r+s-p-q)) = h^0(C,\eta \otimes \mathcal{O}_C(p+q-r-s)).
\]
Consequently, the variety $\Gamma$ is naturally equivariant with respect to the action of the diedral group $D_4$ on $C^4$. Here, the first two generators act by swapping the components within each pair:
\begin{align*}
\iota_1:\Gamma \longrightarrow \Gamma, & \quad (p,q,r,s) \mapsto (q,p,r,s), \\
\iota_2:\Gamma \longrightarrow \Gamma, & \quad (p,q,r,s) \mapsto (p,q,s,r),
\end{align*}
while the third generator exchanges the two pairs $(p,q)$ and $(r,s)$ via
\[
\iota_3:\Gamma \longrightarrow \Gamma, \quad (p,q,r,s) \mapsto (r,s,p,q).
\]
The condition $h^0(C,\eta \otimes \mathcal{O}_C(p+q-r-s))>0$ is clearly invariant under the permutations within each pair, whereas its invariance under the exchange of the two pairs is guaranteed by the aforementioned Riemann--Roch argument.

As a result, the divisor class $\Gamma$ must be invariant under the induced action of this symmetry group on the Néron--Severi group of $C^4$. This forces all fiber classes $F_i$ to share the same coefficient $\alpha$. Furthermore, the diagonal classes split into two distinct orbits under this action, namely $\{\Delta_{12},\Delta_{34}\}$ and $\{\Delta_{13},\Delta_{14},\Delta_{23},\Delta_{24}\}$. The class of $\Gamma$ is thus constrained to be of the form
\[
\Gamma\equiv\alpha\sum_{i=1}^4 F_i+\beta(\Delta_{12}+\Delta_{34}) +\gamma(\Delta_{13}+\Delta_{14}+\Delta_{23}+\Delta_{24}),
\]
for some coefficients $\alpha,\beta,\gamma \in \mathbb{Z}$.

To determine these coefficients, we intersect $\Gamma$ with three independent codimension--$3$ test cycles, namely $F_1 \cdot F_2 \cdot F_3$, $\Delta_{12}\cdot F_3 \cdot F_4$, and $\Delta_{13}\cdot F_2 \cdot F_4$. We evaluate these intersection numbers set-theoretically by interpreting them in terms of the ramification of suitable linear systems:

\begin{enumerate}
    \item \textbf{Intersection with $F_1 \cdot F_2 \cdot F_3$:} Fixing three general points $p,q,r\in C$, this intersection corresponds to the points $(p,q,r,s)$ satisfying $h^0\big(\eta\otimes\mathcal{O}_C(p+q-r-s)\big)>0.$
    Equivalently, $s$ belongs to the support of the unique effective divisor associated with the line bundle $\eta\otimes\mathcal{O}_C(p+q-r)$,
    which has degree $g$. Therefore
    $$\Gamma\cdot F_1\cdot F_2\cdot F_3=g.$$
    
    \item \textbf{Intersection with $\Delta_{12} \cdot F_3 \cdot F_4$:} This intersection with $\Gamma$ corresponds to points of the form $(p,p,r,s)$ such that $h^0\big(\eta \otimes \mathcal{O}_C(2p-r-s)\big)>0$. In particular the linear system $|\eta(r+s)|$ defines $\varphi:C\to \Pset^1$ and the condition $h^0\big(\eta \otimes \mathcal{O}_C(2p-r-s)\big)>0$ gives by simmetry $h^0\big(\eta \otimes \mathcal{O}_C(r+s-2p)\big)>0$, which is equivalent to say that $p$ is a ramification point of $\varphi$. The Hurwitz formula implies that the total number of such ramification points is $4g$, hence
    \[
    \Gamma \cdot \Delta_{12} \cdot F_3 \cdot F_4 = 4g.
    \]
    
    \item \textbf{Intersection with $\Delta_{13} \cdot F_2 \cdot F_4$:} This matches points of the form $(p,q,p,s)$ such that $h^0\big(\eta \otimes \mathcal{O}_C(q-s)\big)>0$, which means that $(q,s)\in \Gamma_\eta$, but this cannot happen for general choices of $q$ and $s$, meaning the intersection is empty:
    \[
    \Gamma \cdot \Delta_{13} \cdot F_2 \cdot F_4 = 0.
    \]
\end{enumerate}
On the other hand, the intersection numbers among the generators are governed by the following standard intersection rules on $C^4$:
\begin{align*}
F_i \cdot F_j \cdot F_k \cdot F_l &=
\begin{cases}
0 & \text{if at least two indices coincide},\\
1 & \text{if } i,j,k,l \text{ are all distinct},
\end{cases} \\[6pt]
F_i \cdot F_j \cdot F_k \cdot \Delta_{lm} &=
\begin{cases}
1 & \text{if } \{l,m\} \not\subset \{i,j,k\},\\
0 & \text{if } \{l,m\} \subset \{i,j,k\},
\end{cases} \\[6pt]
\Delta_{ij} \cdot \Delta_{kl} \cdot F_a \cdot F_b &=
\begin{cases}
0 & \text{if } i,j,k,l \text{ are all distinct},\\
1 & \text{if }\begin{array}{l}
\{u\}=\{i,j\}\cap\{k,l\}\notin\{a,b\},\\
\{v\}=\{1,2,3,4\}\setminus\{i,j,k,l\}\in\{a,b\}
\end{array} \\
2-2g & \text{if } (i,j) = (k,l) \text{ and } \{a,b\} = \{1,\dots,4\} \setminus \{i,j\}.
\end{cases}
\end{align*}
In particular, the self-intersection identity follows from the projection formula applied to $\pi_{ij}$:
\[
(\pi_{ij})_*\big(\pi_{ij}^*(\Delta^2) \cdot F_k \cdot F_l\big) = ((\pi_{ij})_*(F_k \cdot F_l))\cdot \Delta^2 = \Delta^2 = 2-2g,
\]
where $\{i,j,k,l\}=\{1,\dots,4\}$ are distinct. Testing our parameterized class $\Gamma$ against the three test cycles yields the following linear system:
\begin{align*}
\begin{cases}
\Gamma \cdot F_1 \cdot F_2 \cdot F_3 = g \\[3pt]
\Gamma \cdot \Delta_{12} \cdot F_3 \cdot F_4 = 4g \\[3pt]
\Gamma \cdot \Delta_{13} \cdot F_2 \cdot F_4 = 0
\end{cases}
\ \implies \ 
\begin{cases}
\alpha + \beta + 2\gamma = g \\[3pt]
2\alpha + (2-2g)\beta + 4\gamma = 4g \\[3pt]
2\alpha + 2\beta + (4-2g)\gamma = 0
\end{cases}
\ \implies \ 
\begin{cases}
\alpha = g-1 \\[3pt]
\beta = -1 \\[3pt]
\gamma = 1
\end{cases},
\end{align*}
obtaining the expression as in the statement.
\end{proof}
\subsection{Smoothness of the Scorza threefold}

In this subsection we study the smoothness of the threefold $\Gamma$.
The strategy is to reduce the analysis of singular points to ramification properties of suitable projections, using the classical argument of Griffiths-Harris for varieties in higher products.

\medskip


We now prove the following.
\begin{theorem} \label{liscezza3}
For a general $(C,\eta)\in S^+_g$ with $g\geq 2$, the threefold $\Gamma$ is smooth outside the locus $$Z'= \big\{ (p,q,r,s) \vert h^0(\eta\otimes\mathcal{O}_C(p+q-r-s)= 2\big\}.$$
\end{theorem}
\begin{proof}
By Proposition~\ref{criteriogriffiths}, the smoothness of $\Gamma$ at $x$ is equivalent to the existence of at least one subset of indices $I \subset \{1,2,3,4\}$ with $|I|=3$ such that the differential of the projection $\pi_I|_{\Gamma}$ at $x$ is an isomorphism. Consequently, $x$ is a singular point if and only if the differential $d(\pi_I|_\Gamma)_x$ fails to be an isomorphism for all four possible choices of $I$. 

Geometrically, the failure of $d(\pi_I|_\Gamma)_x$ to be an isomorphism means that the projection maps $\pi_I|_\Gamma: \Gamma \to C^3$ are ramified at $x$. In terms of the linear systems on the curve $C$, these infinitesimal conditions translate into the existence of non-zero sections with higher-order zeros at the components of $x$. Specifically, the ramification of the projection that forgets the first factor implies that the point $p$ moves infinitesimally preserving the condition, which corresponds to the appearance of a double point $2p$ in the effective divisor presentation (and similarly for $q, r,$ and $s$). 

Therefore, the simultaneous ramification of all four projections onto triples of factors of $C^4$ yields the following system of conditions on the line bundle $\eta$:
\begin{equation}\label{eq:singularity_system}
\begin{cases}
h^0\big(\eta\otimes\mathcal{O}_C(p+q-2r-s)\big)>0, \\[3pt]
h^0\big(\eta\otimes\mathcal{O}_C(p+q-r-2s)\big)>0, \\[3pt]
h^0\big(\eta\otimes\mathcal{O}_C(r+s-2p-q)\big)>0, \\[3pt]
h^0\big(\eta\otimes\mathcal{O}_C(r+s-p-2q)\big)>0.
\end{cases}
\end{equation}
Let $
x=(p,q,r,s)\in \Gamma\setminus Z',
$
so that
\[
h^0\big(\eta\otimes\mathcal{O}_C(p+q-r-s)\big)\neq 2.
\]
By Riemann--Roch, this implies also
\[
h^0\big(\eta\otimes\mathcal{O}_C(r+s-p-q)\big)\neq 2.
\]
We argue by contradiction: assume that $x$ is singular and let us assume, at first, that $x$ does not belong to any of the diagonals. Then the four conditions in
\eqref{eq:singularity_system} are satisfied. The first two imply that
there exist effective divisors $D_1,D_2$ such that
\begin{equation}\label{eq:D12}
\eta+p+q-r-s\sim D_1+r,\qquad
\eta+p+q-r-s\sim D_2+s,
\end{equation}
while the last two imply the existence of effective divisors $B_1,B_2$
such that
\begin{equation}\label{eq:B12}
\eta+r+s-p-q\sim B_1+p,\qquad
\eta+r+s-p-q\sim B_2+q.
\end{equation}
Therefore
\[
D_1+r\sim D_2+s,\qquad B_1+p\sim B_2+q .
\]
In particular $p\neq q$ and $r\neq s$, by assumption. From the first equivalence in
\eqref{eq:D12} we get $s\in\operatorname{Supp}(D_1)$, while from the
second one we get $r\in\operatorname{Supp}(D_2)$. Similarly,
\eqref{eq:B12} gives
$q\in\operatorname{Supp}(B_1)$ and $p\in\operatorname{Supp}(B_2)$.
Hence we can write
\[
D_1=E_1+s,\qquad D_2=E_2+r,\qquad
B_1=F_1+q,\qquad B_2=F_2+p
\]
for some effective divisors $E_1,E_2,F_1,F_2$.

Using \eqref{eq:D12} and \eqref{eq:B12}, we obtain
\begin{equation}\label{eq:EF}
\eta+p+q\sim E_1+2r+2s,\qquad
\eta+r+s\sim F_1+2p+2q .
\end{equation}
Subtracting the two equivalences in \eqref{eq:EF}, we get
\[
F_1+3p+3q\sim E_1+3r+3s .
\]
We can apply Proposition~\ref{prop fanzo} to the two divisors
\[
F_1+3p+3q,\qquad E_1+3r+3s .
\]
Indeed, they are distinct: otherwise we would obtain
$\eta\sim E_1+r+s-p-q$, which would imply that $\eta$ is effective,
contradicting the assumption that $(C,\eta)$ is an ineffective spin curve.

Therefore, by Proposition~\ref{prop fanzo}, setting
\[
\widetilde E=F_1+3p+3q-E_1-3r-3s,
\]
we have
\begin{equation}\label{eq:fanzo_conclusion}
h^0\big(K_C-\operatorname{Supp}(\widetilde E)\big)=0 .
\end{equation}
On the other hand, adding the two equivalences in \eqref{eq:EF} gives
\[
K_C\sim E_1+F_1+p+q+r+s .
\]
Since
\[
\operatorname{Supp}(\widetilde E)\subset
\operatorname{Supp}(E_1)\cup\operatorname{Supp}(F_1)\cup\{p,q,r,s\},
\]
it follows that
\[
h^0\big(K_C-\operatorname{Supp}(\widetilde E)\big)>0,
\]
contradicting \eqref{eq:fanzo_conclusion}.

Now we focus on the case where at least two entries of $x$ coincide. 
Note that the cases $(p,p,p,p)$ and $(p,p,p,s)$ (and the symmetric cases) are ruled out since they imply effectiveness for $\eta$, we have to check (up to simmetry): 
\begin{enumerate}
\item $x=(p,q,p,s)$, i.e. $x\in \Delta_{13}$;
\item $x=(p,p,r,s)$, i.e. $x\in \Delta_{12}$;
\item $x=(p,p,r,r)$, i.e. the intersection of diagonals.
\end{enumerate}
All these are ruled out by expressing particular effective divisors and applying them \ref{prop fanzo}. We show just the first case, since the other two are very similar.

Let $x=(p,q,p,s)\in \Gamma$. Then $
h^0\big(\eta\otimes\mathcal O_C(q-s)\big)>0$.
The line bundle
$\eta\otimes\mathcal O_C(q-s)$ has a unique section. In particular,
$x\notin Z'$.

Assume by contradiction that $x$ is singular. Then the conditions in
\eqref{eq:singularity_system} give effective divisors $D_1,D_2,E_1,E_2$
such that
\begin{equation}\label{eq:diag13_conditions}
\begin{cases}
\eta\otimes\mathcal O_C(q-p-s)\sim D_1,\\[3pt]
\eta\otimes\mathcal O_C(q-2s)\sim D_2,\\[3pt]
\eta\otimes\mathcal O_C(s-p-q)\sim E_1,\\[3pt]
\eta\otimes\mathcal O_C(s-2q)\sim E_2 .
\end{cases}
\end{equation}

Since the linear systems $|\eta\otimes \mathcal{O}_C(q)|$ and $|\eta\otimes\mathcal{O}_C(s)|$ consist of a unique
effective divisor, the first two equivalences in
\eqref{eq:diag13_conditions} imply
\[
D_1+p+s\sim D_2+2s,
\]
and hence
\[
D_1+p\sim D_2+s.
\]
Similarly, the last two equivalences imply
\[
E_1+p\sim E_2+q.
\]
Notice that $p\neq s$ and $p\neq q$.
It follows that
\[
s\in\operatorname{Supp}(D_1),\qquad
p\in\operatorname{Supp}(D_2),\qquad
q\in\operatorname{Supp}(E_1),\qquad
p\in\operatorname{Supp}(E_2).
\]
Hence there exist effective divisors $\widetilde D,\widetilde E$ such that
\begin{equation}\label{eq:diag13_decomposition}
\begin{cases}
\eta+q\sim \widetilde D+p+2s,\\[3pt]
\eta+s\sim \widetilde E+p+2q .
\end{cases}
\end{equation}

Subtracting the two equivalences in \eqref{eq:diag13_decomposition}, we get
\[
\widetilde D+3s\sim \widetilde E+3q .
\]
The two divisors above are distinct, otherwise we would obtain an effective
divisor linearly equivalent to $\eta$, contradicting the ineffectivity of
the spin structure. Therefore Proposition~\ref{prop fanzo} applies and, if
we set
\[
\widetilde F=\widetilde D-\widetilde E+3s-3q,
\]
we obtain
\begin{equation}\label{eq:diag13_fanzo}
h^0\big(K_C-\operatorname{Supp}(\widetilde F)\big)=0 .
\end{equation}
On the other hand, adding the two equivalences in
\eqref{eq:diag13_decomposition} and using $2\eta\cong K_C$, we get
\[
K_C\sim \widetilde D+\widetilde E+2p+q+s .
\]
Since
\[
\operatorname{Supp}(\widetilde F)\subset
\operatorname{Supp}(\widetilde D)\cup
\operatorname{Supp}(\widetilde E)\cup\{q,s\},
\]
it follows that
\[
h^0\big(K_C-\operatorname{Supp}(\widetilde F)-2p\big)>0,
\]
and hence
\[
h^0\big(K_C-\operatorname{Supp}(\widetilde F)\big)>0,
\]
contradicting \eqref{eq:diag13_fanzo}. Thus no point of
$\Delta_{13}$ is singular. The same argument applies to the diagonals
$\Delta_{14}, \Delta_{23}$ and $\Delta_{24}$ by simmetry.
\\It remains to consider the diagonals $\Delta_{12}$ and $\Delta_{34}$.
Let $x=(p,p,r,s)\in\Gamma$. Then
\[
h^0\big(\eta\otimes\mathcal O_C(2p-r-s)\big)>0 .
\]
Assuming $x\notin Z'$, this space has dimension one. If $x$ were singular,
then the conditions in \eqref{eq:singularity_system} reduce to
\begin{equation}\label{eq:diag12_conditions}
\begin{cases}
h^0\big(\eta\otimes\mathcal O_C(2p-2r-s)\big)>0,\\[3pt]
h^0\big(\eta\otimes\mathcal O_C(2p-r-2s)\big)>0,\\[3pt]
h^0\big(\eta\otimes\mathcal O_C(r+s-3p)\big)>0 .
\end{cases}
\end{equation}
The same divisor argument as above, applied to the effective divisors
arising from \eqref{eq:diag12_conditions}, produces two distinct linearly
equivalent divisors of the form required in Proposition~\ref{prop fanzo},
and hence leads to the same contradiction. Therefore $\Delta_{12}$ contains
no singular points. By symmetry the same holds for $\Delta_{34}$.

The last case produces analogous computations. Hence every point of
$\Gamma\setminus Z'$ lying on a diagonal of $C^4$ is smooth as claimed.
\end{proof}

\bibliographystyle{alpha} 
\bibliography{bibliografia} 

@article{FarkasVerra2014,
  title={The geometry of the moduli space of odd spin curves},
  author={Farkas, Gavril and Verra, Alessandro},
  journal={Annals of Mathematics},
  volume={180},
  number={3},
  pages={927--970},
  year={2014},
  publisher={JSTOR}
}

@article{DolgachevKanev1993,
  title={Polar covariants of plane cubics and quartics},
  author={Dolgachev, Igor and Kanev, Vassil},
  journal={Advances in Mathematics},
  volume={98},
  number={2},
  pages={216--301},
  year={1993},
  publisher={Elsevier}
}

@article{FarkasIzadi2024,
  title={Szeg{\"o} kernels and Scorza quartics on the moduli space of spin curves},
  author={Farkas, Gavril and Izadi, Elham},
  journal={arXiv:2409.13303},
  year={2024}
}

@article{Griffiths1983,
  author = {Griffiths, Phillip A.},
  title = {Infinitesimal variations of hodge structure {(III)} : determinantal varieties and the infinitesimal invariant of normal functions},
  journal = {Compositio Mathematica},
  pages = {267--324},
  publisher = {Martinus Nijhoff Publishers},
  volume = {50},
  number = {2-3},
  year = {1983},
  mrnumber = {720290},
  zbl = {0576.14009},
  language = {en},
  url = {https://www.numdam.org/item/CM_1983__50_2-3_267_0/}
}

@book{GriffithsHarris1978,
  AUTHOR = {Griffiths, Phillip and Harris, Joseph},
  TITLE = {Principles of algebraic geometry},
  SERIES = {Pure and Applied Mathematics},
  PUBLISHER = {Wiley-Interscience [John Wiley \& Sons], New York},
  YEAR = {1978},
  PAGES = {xii+813},
  ISBN = {0-471-32792-1},
  MRCLASS = {14-01},
  MRNUMBER = {507725},
  MRREVIEWER = {Gerhard Pfister},
}

@incollection{Cornalba1989,
  AUTHOR = {Cornalba, Maurizio},
  TITLE = {Moduli of curves and theta-characteristics},
  BOOKTITLE = {Lectures on {R}iemann surfaces ({T}rieste, 1987)},
  PAGES = {560--589},
  PUBLISHER = {World Sci. Publ., Teaneck, NJ},
  YEAR = {1989},
  ISBN = {9971-50-902-4},
  MRCLASS = {14H10 (14H42)},
  MRNUMBER = {1082361},
  MRREVIEWER = {Olivier Debarre},
}

@article{Farkas2010,
  AUTHOR = {Farkas, Gavril},
  TITLE = {The birational type of the moduli space of even spin curves},
  JOURNAL = {Adv. Math.},
  FJOURNAL = {Advances in Mathematics},
  VOLUME = {223},
  YEAR = {2010},
  NUMBER = {2},
  PAGES = {433--443},
  ISSN = {0001-8708,1090-2082},
  MRCLASS = {14H10 (14E08)},
  MRNUMBER = {2565536},
  MRREVIEWER = {Montserrat Teixidor i Bigas},
  DOI = {10.1016/j.aim.2009.08.011},
  URL = {https://doi.org/10.1016/j.aim.2009.08.011},
}

@article{Scorza1899,
  title={Sopra la teoria delle figure polari delle curve piane del 4° ordine},
  author={Scorza, Gaetano},
  journal={Annali di Matematica Pura ed Applicata},
  volume={2},
  number={3},
  pages={155--202},
  year={1899},
  publisher={Springer}
}

@article{Scorza1900,
  author = {Scorza, Gaetano},
  title = {Sopra le curve canoniche di uno spazio lineaire qualunque e sopra certi loro covarianti quartici},
  journal = {Atti della Reale Accademia delle Scienze di Torino},
  volume = {35},
  pages = {765--773},
  year = {1900}
}

@article{FassinaPirola2026,
  author = {Fassina, Lorenzo and Pirola, Gian Pietro},
  title = {A few remarks on sections of the Picard bundle of family of curves},
  journal = {arXiv:2602.14888},
  year = {2026},
  eprint = {2602.14888},
  archivePrefix = {arXiv},
  primaryClass = {math.AG},
  url = {https://arxiv.org/abs/2602.14888}
}

@article{HowardSommese1983,
  AUTHOR = {Howard, Alan and Sommese, Andrew J.},
  TITLE = {On the theorem of de {F}ranchis},
  JOURNAL = {Ann. Scuola Norm. Sup. Pisa Cl. Sci. (4)},
  FJOURNAL = {Annali della Scuola Normale Superiore di Pisa. Classe di Scienze. Serie IV},
  VOLUME = {10},
  YEAR = {1983},
  NUMBER = {3},
  PAGES = {429--436},
  ISSN = {0391-173X,2036-2145},
  MRCLASS = {32H99 (14E99)},
  MRNUMBER = {739918},
  MRREVIEWER = {Alexandru Mihai},
  URL = {http://www.numdam.org/item?id=ASNSP_1983_4_10_3_429_0},
}

@article{GrushevskySalvatiManni2010,
  title={The Scorza correspondence in genus 3},
  author={Grushevsky, Samuel and Salvati Manni, Riccardo},
  journal={arXiv:1009.0375},
  year={2010},
  eprint={1009.0375},
  archivePrefix={arXiv},
  primaryClass={math.AG},
  url={https://arxiv.org/abs/1009.0375},
}

@article{Koizumi1976,
  author={Koizumi, Shoji},
  title={The Ring of Algebraic Correspondences on a Generic Curve of Genus g},
  journal={Nagoya Mathematical Journal},
  volume={60},
  year={1976},
  pages={173--180},
  doi={10.1017/S0027763000017219}
}

@unpublished{AKP26,
  author = {Agostini, Daniele and Kummer, Mario and Park, Jinhyung},
  title  = {{Ulrich} sheaves and determinantal representations for higher
            secant varieties of curves},
  note   = {Preprint},
  year   = {2026}
}

\end{document}